\documentclass{amsart}

\usepackage{hyperref}

\usepackage{amssymb}
\usepackage{amsmath}
\usepackage{amsthm}
\usepackage{enumerate}
\usepackage{graphicx}
\usepackage{tikz-cd}

\theoremstyle{plain}

\newtheorem{theorem}{Theorem}[section]

\newtheorem{lemma}[theorem]{Lemma}
\newtheorem{corollary}[theorem]{Corollary}
\newtheorem{question}[theorem]{Question}

\theoremstyle{definition}
\newtheorem{definition}[theorem]{Definition}

\theoremstyle{remark}
\newtheorem*{remark}{Remark}

\newtheorem*{acknowledgments}{Acknowledgments}

\renewcommand {\tilde} {\widetilde}

\begin{document}

\title{On the analytic and Cauchy capacities}

\date{November 26, 2014}

\author[M. Younsi]{Malik Younsi}
\thanks{Supported by NSERC}
\address{Department of Mathematics, Stony Brook University, Stony Brook, NY 11794-3651, United States.}
\email{malik.younsi@stonybrook.edu}


\begin{abstract}
We give new sufficient conditions for a compact set $E \subseteq \mathbb{C}$ to satisfy $\gamma(E)=\gamma_c(E)$, where $\gamma$ is the analytic capacity and $\gamma_c$ is the Cauchy capacity. As a consequence, we provide examples of compact plane sets such that the above equality holds but the Ahlfors function is not the Cauchy transform of any complex Borel measure supported on the set.
\end{abstract}

\maketitle

\section{Introduction}
Let $E$ be a compact subset of the complex plane $\mathbb{C}$. The \textit{analytic capacity} of $E$ is defined by
$$\gamma(E):=\sup \{|f'(\infty)|: f \in H^{\infty}(\Omega), |f| \leq 1\},$$
where $\Omega$ is the unbounded component of $\mathbb{C}_\infty \setminus E$, the complement of $E$ with respect to the Riemann sphere, $H^{\infty}(\Omega)$ is the class of all bounded holomorphic functions on $\Omega$ and $f'(\infty):=\lim_{z \to \infty}z(f(z)-f(\infty))$.
\\

Analytic capacity was first introduced by Ahlfors in his celebrated paper \cite{AHL} for the study of a problem generally attributed to Painlev\'e in 1888 asking to find a geometric characterization of the compact sets that are removable for bounded holomorphic functions. It was later observed by Vitushkin \cite{VIT} that analytic capacity is a fundamental tool in the theory of uniform rational approximation of holomorphic functions.
\\

It follows easily from the definition that analytic capacity is monotonic, i.e. $\gamma(E) \leq \gamma(F)$ whenever $E \subseteq F$, and that analytic capacity is \textit{outer regular} in the sense that if $E_1 \supseteq E_2 \supseteq \dots$ is a decreasing sequence of compact sets, then $\gamma(E_k) \to \gamma(\cap_n E_n)$ as $k \to \infty$. Furthermore, it is well-known that for any compact set $E$ of positive analytic capacity, there exist a unique function $f \in H^{\infty}(\Omega)$ with $|f| \leq 1$ and $f'(\infty)=\gamma(E)$, called the \textit{Ahlfors function for} $E$ or \textit{on} $\Omega$. Note that the extremality of $f$ implies that it must vanish at the point $\infty$. By convention, the Ahlfors function is defined to be identically zero on each bounded component of $\mathbb{C}_\infty \setminus E$. We also mention that in some particular cases, the properties of the Ahlfors function are well-known. For instance, if $\Omega$ is a nondegenerate $m$-connected domain, then the Ahlfors function $f$ is a degree $m$ proper holomorphic map of $\Omega$ onto $\mathbb{D}$. In particular, if $E$ is connected and contains more than one point, it is simply the Riemann map, normalized so that $f(\infty)=0$ and $f'(\infty)>0$. For more information on the elementary properties of analytic capacity and Ahlfors functions, we refer the reader to \cite{GAR} and \cite{TOL}.
\\

Following its emergence in 1947, analytic capacity quickly acquired the reputation to be quite difficult to study and its properties have remained mysterious for several decades. The main recent advances in the field are due to Tolsa \cite{TOL3}, who proved that analytic capacity is comparable to a quantity which is easier to comprehend since it is more suitable to real analysis tools. More precisely, define the capacity $\gamma_{+}$ of a compact set $E$ by
$$\gamma_{+}(E):=\sup \{\mu(E) : \operatorname{supp}{\mu} \subseteq E, |\mathcal{C}_\mu| \leq 1 \, \, \mbox{on}\, \, \mathbb{C}_\infty \setminus E\},$$
the supremum being taken over all positive Radon measures $\mu$ supported on $E$ such that the \textit{Cauchy transform}
$$\mathcal{C}_\mu (z):=\int \frac{1}{\xi - z} d\mu(\xi)$$
is bounded by one in modulus outside $E$. Note that for any such measure, $\mathcal{C}_\mu$ is analytic outside $E$ with $\mathcal{C}_\mu'(\infty)=-\mu(E)$, thus we have $\gamma_{+}(E) \leq \gamma(E)$. Tolsa's remarkable result states that $\gamma(E) \leq C \gamma_{+}(E)$ for some universal constant $C$. This theorem has several important consequences. For instance, it gives a complete solution to Painlev\'e's problem for arbitrary compact sets involving the notion of curvature of a measure introduced by Melnikov \cite{MEL}. A previous solution for sets of finite length was obtained earlier by David \cite{DAV}. Moreover, since $\gamma_{+}$ was previously shown by Tolsa to be comparable with a quantity that is subadditive, it follows that $\gamma$ is semi-additive, meaning that there is a universal constant $C'$ such that $\gamma(E \cup F) \leq C'(\gamma(E)+\gamma(F))$ for all compact sets $E,F$. This solved a very difficult problem raised by Vitushkin in 1967. The interested reader may consult \cite{TOL} for more details. We mention in passing that it is not known whether analytic capacity is subadditive; in other words, if $C'$ can be taken equal to 1. See \cite{YOU} for more information on this problem.
\\

A closely related concept is the so-called \textit{Cauchy capacity} of $E$, noted by $\gamma_c(E)$ and defined as
$$\gamma_c(E):= \sup \{|\mu(E)| : \operatorname{supp}{\mu} \subseteq E, |\mathcal{C}_\mu| \leq 1 \, \, \mbox{on}\, \, \mathbb{C}_\infty \setminus E\},$$
where $\mu$ is a finite complex Borel measure supported on $E$. Apparently, the term ``Cauchy capacity'' was used for the first time by Havinson in \cite{HAV}. Clearly, $\gamma_c(E) \leq \gamma(E) \leq C \gamma_{c}(E)$ for all compact sets $E$, where $C$ is the comparability constant in Tolsa's result. In particular, it follows that $\gamma(E)=0$ if and only if $\gamma_{c}(E)=0$, which is quite nontrivial. Our main motivation for the present paper is the study of the following open question :

\begin{question}
\label{prob1}
Is the analytic capacity actually equal to the Cauchy capacity? In other words, is it true that
\begin{equation}
\label{eqprob1}
\gamma(E)=\gamma_c(E)
\end{equation}
for all compact sets $E \subseteq \mathbb{C}$?
\end{question}

Apparently Question \ref{prob1} was raised for the first time by Murai (see \cite[Section 3]{MUR}). It was later studied by Havinson in \cite{HAV2} and \cite{HAV3}. See also \cite[Section 5]{TOL2}.
\\

Equality (\ref{eqprob1}) is known to hold only in some very special cases, such as compact sets of finite painlev\'e length. We say that a compact set $E$ has \textit{finite Painlev\'e length} if there is a number $l$ such that every open set $U$ with $E \subseteq U$ contains a cycle $\Gamma$ surrounding $E$ that consists of finitely many disjoint analytic Jordan curves and has length less than $l$. The infimum of such numbers $l$ is called the \textit{Painlev\'e length} of $E$. The following theorem essentially follows from Cauchy's integral formula and a limiting argument.

\begin{theorem}[Havinson \cite{HAV4}]
\label{ThmPainleve}
Suppose that $E$ has finite Painlev\'e length $k$. If $f$ is any bounded holomorphic function on $\mathbb{C}_\infty \setminus E$ with $f(\infty)=0$, then there is a complex measure $\mu$ supported on $E$ such that $\|\mu\| \leq k \|f\|_{\infty}/2\pi$ and
$$f(z)=\mathcal{C}_\mu(z) \qquad (z \in \mathbb{C}_\infty \setminus E).$$
\end{theorem}
See also \cite[Theorem 3.1, Chapter 2]{GAR}.
\\

In particular, applying the above result to the Ahlfors function for $E$, we obtain
\begin{corollary}
\label{coro1}
If $E$ has finite Painlev\'e length, then
$$\gamma(E)=\gamma_{c}(E).$$
\end{corollary}

A consequence of Corollary \ref{coro1} is that a positive answer to Question \ref{prob1} would follow if one could prove that the Cauchy capacity $\gamma_c$ is outer regular, because every compact subset of the plane can be obtained as a decreasing sequence of compact sets with finite Painlev\'e length.
\\

In this paper, we prove the following result, which can be viewed as a generalization of Corollary \ref{coro1} to sets of $\sigma$-finite Painlev\'e length.

\begin{theorem}
\label{thm2}
Let $E \subseteq \mathbb{C}$ be compact and suppose that there exist a sequence $(E_k)_{k \in \mathbb{N}}$ of compact subsets of $E$ with the following properties :

\begin{enumerate}[\rm(i)]
\item [(i)] every $E_k$ has finite Painlev\'e length;
\item [(ii)] there exist an integer $m$ such that $\Omega$ and every $\Omega_k$ are nondegenerate $m$-connected domains, where $\Omega_k$ and $\Omega$ are the unbounded components of $\mathbb{C}_\infty \setminus E_k$ and $\mathbb{C}_\infty \setminus E$ respectively;
\item [(iii)] the sequence of domains $(\Omega_k)_{k \in \mathbb{N}}$ converges to $\Omega$ in the sense of Carath\'eodory.
\end{enumerate}
Then $\gamma(E)=\gamma_c(E)$.
\end{theorem}

Note that in Theorem \ref{ThmPainleve}, not only the Ahlfors function but every bounded holomorphic function on $\mathbb{C}_\infty \setminus E$ vanishing at infinity is the Cauchy transform of a complex measure supported on $E$. From the point of view of Question \ref{prob1}, a more interesting question is whether the Ahlfors function can always be expressed as the Cauchy transform of a complex measure. This was answered in the negative by Samokhin.

\begin{theorem}[Samokhin \cite{SAM}]
There exist a connected compact set $F$ with connected complement such that the Ahlfors function for $F$ is not the Cauchy transform of any complex measure supported on $F$.
\end{theorem}

In particular, $\gamma_{+}(F) < \gamma(F)$ by uniqueness of the Ahlfors function and by the fact that the capacity $\gamma_{+}$ of a compact set is always attained by some measure. It is not clear whether the latter is true if $\gamma_{+}$ is replaced by $\gamma_{c}$.

\begin{question}
\label{prob2}
Is it true that for every compact set $E$, there exist a complex Borel measure $\mu$ supported on $E$ such that
$$|\mathcal{C}_\mu(z)| \leq 1 \qquad (z \in \mathbb{C}_\infty \setminus E)$$ and
$\mu(E)=\gamma_c(E)$?
\end{question}

We shall give a negative answer to Question \ref{prob2} by proving the following result.

\begin{theorem}
\label{thm0}
There exist a connected compact set $E$ with connected complement such that $\gamma(E)=\gamma_c(E)$ but the Ahlfors function for $E$ is not the Cauchy transform of any complex Borel measure supported on $E$.
\end{theorem}

Using Theorem \ref{thm2}, we can also obtain examples for any prescribed number of connected components.

\begin{theorem}
\label{thm1}
For any $m \in \mathbb{N}$, there exist a compact set $E^m$ with $m$ nondegenerate components such that $\gamma(E^m)=\gamma_c(E^m)$ but the Ahlfors function for $E^m$ is not the Cauchy transform of any complex Borel measure supported on $E^m$.
\end{theorem}

Our construction is a bit simpler than the one in \cite{SAM}, although the latter can be generalized to obtain an example of a simply connected domain $\Omega$ such that for fairly general functionals on $H^{\infty}(\Omega)$, no extremal function can be represented as a Cauchy transform.
\\

The rest of the paper is organized as follows. In Sect. \ref{sec1}, we give a proof of Theorem \ref{thm0} based on a convergence lemma for the analytic capacity of two disjoint closed disks. Section \ref{sec2} contains the proof of a convergence theorem for Ahlfors functions based on the Carath\'eodory kernel convergence theorem for finitely connected domains and on Koebe's circle domain theorem. This convergence theorem is then used in Sect. \ref{sec3} to prove Theorem \ref{thm2}. Finally, Section \ref{sec4} is dedicated to the construction of the sets $E^m$ of Theorem \ref{thm1}.

\section{Proof of Theorem \ref{thm0}}
\label{sec1}
In this section, we prove Theorem \ref{thm0} by constructing a connected compact set $E$ with connected complement such that $\gamma(E)=\gamma_c(E)$, but the Ahlfors function is not the Cauchy transform of any complex Borel measure supported on $E$.
\\

First, we need some preliminaries on convergence in the sense of Carath\'eodory. We shall assume for the remaining of the section that $(\Omega_k)$ is a sequence of domains in the Riemann sphere, each containing the point $\infty$.

\begin{definition}
The \textit{kernel} of the sequence $(\Omega_k)$ (with respect to the point $\infty$) is defined to be the largest domain $\Omega$ containing the point $\infty$ such that if $K$ is a compact subset of $\Omega$, then there exist a $k_0$ such that $K \subseteq \Omega_k$ for all $k \geq k_0$, if such a domain exists. If not, then we say that the kernel of $(\Omega_k)$ does not exist.
\\

Furthermore, we say that a sequence $(\Omega_k)$ converges to $\Omega$ (in the sense of Carath\'eodory) if $\Omega$ is the kernel of every subsequence of $(\Omega_k)$. This is denoted by $\Omega_k \to \Omega$.
\end{definition}

Now, for $k \geq 1$, let $g_k$ be univalent on $\Omega_k$ and \textit{normalized at infinity}, meaning that
\begin{equation}
\label{norm}
g_k(z)=z+\frac{a_1}{z}+\frac{a_2}{z^2}+\dots
\end{equation}
in a neighborhood of the point $\infty$ or, equivalently, that $\lim_{z \to \infty} (g_k(z)-z)=0$. Note that the only M\"{o}bius transformation normalized at infinity is the identity.
\\

The following two results can be viewed as generalizations of the fact that the family of normalized Schlicht functions on the unit disk is normal and of the Carath\'eodory kernel convergence theorem for simply connected domains.

\begin{lemma}
\label{lem normal}
Suppose that the kernel of $(\Omega_k)$ exists and denote it by $\Omega$. Then there exist a subsequence $(g_{k_l})_{l \in \mathbb{N}}$ such that the kernel of $(\Omega_{k_l})$ is $\Omega$ and $(g_{k_l})_{l \in \mathbb{N}}$ converges locally uniformly to a univalent function $g$ on $\Omega$ normalized at infinity.
\end{lemma}

\begin{theorem}[Generalized Carath\'eodory kernel convergence theorem]
\label{thm cara}
Suppose that $\Omega_k \to \Omega$. Then $(g_k)_{k \in \mathbb{N}}$ converges locally uniformly on $\Omega$ to a univalent function $g$ normalized at infinity if and only if $(g_k(\Omega_k))$ converges to a domain $\tilde{\Omega}$. If this is the case, then $\tilde{\Omega}=g(\Omega)$ and $g_k^{-1} \to g^{-1}$ locally uniformly on $\tilde{\Omega}$.
\end{theorem}

For a proof of these results, see \cite[Section 15.4]{CON}.
\\

For the rest of the paper, we will be interested in the case where each $\Omega_k$ is a nondegenerate $m$-connected domain, for some fixed $m \in \mathbb{N}$. The following theorem of Koebe states that every such domain is conformally equivalent to a \textit{nondegenerate} $m$-\textit{connected circle domain}, that is, a domain whose complement is a union of $m$ disjoint closed round disks.

\begin{theorem}[Koebe's circle domain theorem]
Let $\Omega$ be a nondegenerate finitely connected domain. Then there exist a conformal map $g:\Omega \to \Omega'$, where $\Omega'$ is a nondegenerate circle domain. Moreover, if $g_1$ is another conformal map of $\Omega$ onto a nondegenerate circle domain, then $g_1=M \circ g$ for some M\"{o}bius transformation $M$.
\end{theorem}

It follows that for any nondegenerate finitely connected domain $\Omega$, there exist a unique Koebe map $g:\Omega \to \Omega'$ normalized at infinity. This conformal map $g$ is called the \textit{normalized Koebe map of} $\Omega$.
\\

We shall also need the following conformal representation result for nondegenerate doubly connected circle domains.

\begin{lemma}
\label{doubly}
Let $\Omega'$ be a nondegenerate doubly connected circle domain. Assume moreover that the circles bounding $\Omega'$ are centered on the real axis. Then there exist a unique conformal map $h:\Omega' \to \Omega''$ normalized at infinity, where $\Omega''$ is the complement in $\mathbb{C}_\infty$ of two disjoint closed intervals contained in the same horizontal line.
\end{lemma}

Such a domain $\Omega''$ is called a \textit{doubly connected horizontal slit domain} and the map $h:\Omega' \to \Omega''$ is called the \textit{normalized slit map of} $\Omega$.

\begin{proof}
First note that $\Omega' \cap \mathbb{H}$ is a Jordan domain, where $\mathbb{H}$ is the upper half-plane. Let $\phi:\overline{\Omega' \cap \mathbb{H}} \to \overline{\mathbb{H}}$ be a homeomorphism conformal on $\Omega' \cap \mathbb{H}$ with $\phi(\infty)=\infty$. By the Schwarz reflection principle, the map $\phi$ can be extended to a conformal map  $\psi:\Omega' \to D$, where $D$ is the complement of two disjoint closed intervals in the real axis. Note that the limit $\lim_{z \to \infty} \psi(z)/z$ is positive, since $\phi$ must preserve the orientation of the boundary. Composing $\psi$ with an appropriate linear function yields a conformal map $h:\Omega' \to \Omega''$ normalized at infinity, where $\Omega''$ is a doubly connected horizontal slit domain.
\\

Finally, the uniqueness of the normalized slit map $h$ is well-known, see e.g. \cite[Section 2, Chapter 5]{GOL}.

\end{proof}

We can now prove the following convergence lemma for the analytic capacity of two disjoint closed disks.

\begin{lemma}
\label{lem 2disks}
For $k \in \mathbb{N}$, let $F_k$ be a union of two disjoint closed disks. Suppose that the centers and radii of the disks converge to $c_1,c_2$ and $r_1,r_2$ respectively, where the closed disks $\overline{\mathbb{D}}(c_1,r_1)$ and $\overline{\mathbb{D}}(c_2,r_2)$ intersect at exactly one point. Then
$$\gamma(F_k) \to \gamma(F),$$
where $F:=\overline{\mathbb{D}}(c_1,r_1) \cup \overline{\mathbb{D}}(c_2,r_2)$.
\end{lemma}

\begin{proof}
Translating and rotating if necessary, we can assume that the disks are centered on the real axis. Then each domain $\Omega_k':=\mathbb{C}_\infty \setminus F_k$ is symmetric with respect to the real axis. For $k \in \mathbb{N}$, let $h_k: \Omega_k' \to \Omega_k''$ be the normalized slit map of $\Omega_k'$. We shall prove that $\gamma(F_k) \to \gamma(F)$ by showing that every subsequence of $(\gamma(F_k))_{k \in \mathbb{N}}$ has a subsequence that converges to $\gamma(F)$.
\\

Indeed, first note that $\Omega_k' \to \Omega'$, where $\Omega':=\mathbb{C}_\infty \setminus F$. Let $(\gamma(F_k))_{k \in S}$ be a subsequence. Then by Lemma \ref{lem normal}, the corresponding sequence of normalized slit maps $(h_k)_{k \in S}$ has a subsequence $(h_k)_{k \in S'}$ that converges locally uniformly to a univalent function $h$ on $\Omega'$ normalized at infinity. By Theorem \ref{thm cara}, the corresponding subsequence of doubly connected horizontal slit domains $(\Omega_k'')_{k \in S'}$ converges to $h(\Omega')$ in the sense of Carath\'eodory. Now, the sequences of centers and diameters of the intervals bounding $\Omega_k''$ for $k \in S'$ must be bounded, otherwise the kernel of $(\Omega_k'')_{k \in S'}$ would not contain a neighborhood of infinity. Thus, passing to a subsequence if necessary, we can assume that the intervals bounding $\Omega_k''$ for $k \in S'$ converge to two closed intervals $I_1, I_2$ belonging to the same horizontal line, as $k \to \infty$ in $S'$.
\\

Now, it is easy to see that $h(\Omega')$ must be the domain bounded by the two intervals $I_1$ and $I_2$. Since $\Omega'$ is simply connected and $h$ is univalent, $h(\Omega')$ must also be simply connected and so the only possibility is that the two intervals intersect at exactly one point. By a result of Pommerenke, the analytic capacity of a linear compact set is equal to a quarter of its length (see e.g. \cite[Theorem 6.2, Chapter 1]{GAR}). Since the lengths of the intervals bounding $\Omega_k''$ for $k \in S'$ converge to the lengths of $I_1$ and $I_2$, it follows that
$$\gamma(\mathbb{C}_\infty \setminus \Omega_k'') \to \gamma(I_1 \cup I_2) = \gamma(\mathbb{C}_\infty \setminus h(\Omega')),$$
as $k \to \infty$ in $S'$. Finally, a simple change of variable argument shows that
$$\gamma(\mathbb{C}_\infty \setminus \Omega_k'') = \gamma(\mathbb{C}_\infty \setminus h_k(\Omega_k'))=\gamma(\mathbb{C}_\infty \setminus \Omega_k')=\gamma(F_k)$$
and
$$\gamma(\mathbb{C}_\infty \setminus h(\Omega')) = \gamma(\mathbb{C}_\infty \setminus \Omega')= \gamma(F),$$
so that $\gamma(F_k) \to \gamma(F)$ as $k \to \infty$ in $S'$.
\\

This completes the proof of the lemma.
\end{proof}

We can now prove Theorem \ref{thm0}. More precisely, let us construct a connected compact set $E$ with connected complement such that $\gamma(E)=\gamma_c(E)$, but the Ahlfors function is not the Cauchy transform of any complex Borel measure supported on $E$.

\begin{proof}
Let $E$ be the union of the nonrectifiable curve
$$\Gamma:=\{x+ix\operatorname{sin}(1/x): x\in (0,1/\pi]\}$$
with the line segment $[-i,i]$. Then $E$ is a connected compact set and $\Omega:=\mathbb{C}_\infty \setminus E$ is connected. Let $f$ be the Ahlfors function on $\Omega$.

\begin{figure}[h!t!b]
\begin{center}
\includegraphics[width=6cm, height=6cm]{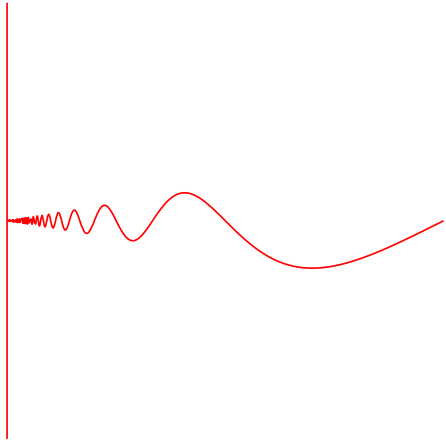}
\caption{The compact set $E$.}
\end{center}
\end{figure}

First, we show that $f$ is not the Cauchy transform of any complex Borel measure supported on $E$. The proof of this is similar to the one in \cite{SAM}. Suppose, in order to obtain a contradiction, that $f=\mathcal{C}\mu$ for such a measure $\mu$. Let $\Omega_{+}$ and $\Omega_{-}$ denote the upper and lower parts respectively of the complement of $\Gamma$ in the strip $\{z: 0<\operatorname{Re}z<1/\pi\}$. For $k \in \mathbb{N}$, let $R_k$ be the open rectangle
$$R_k:= \left\{z=x+iy : \frac{1}{\pi+k} <x<1/2, -1<y<1\right\}.$$
\newpage

\begin{figure}[h!t!b]
\begin{center}
\includegraphics[width=6cm, height=6cm]{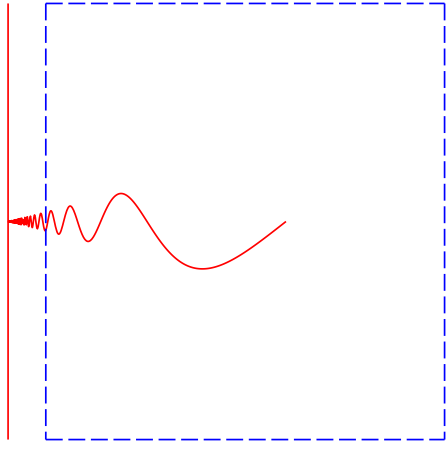}
\caption{The rectangle $R_k$.}
\end{center}
\end{figure}

It is easy to see that the restriction of $f$ to $R_k \setminus \Gamma$ is the Cauchy transform of a measure $\mu_k$ defined by $(1/2\pi i) f(\zeta) d\zeta$ on the linear pieces of the boundary and by $(1/2\pi i) (f_{+}(\zeta)-f_{-}(\zeta))d\zeta$ on $\Gamma \cap R_k$, where $f_{+}$ and $f_{-}$ are the boundary values of $f$ from inside $\Omega_{+}$ and $\Omega_{-}$ respectively, i.e.
$$f_{+}(\zeta) = \lim_{z \to \zeta, z \in \Omega_{+}}f(z)$$
and
$$f_{-}(\zeta) = \lim_{z \to \zeta, z \in \Omega_{-}}f(z).$$
It follows that on $R_k \setminus \Gamma$, we have
$$\mathcal{C}(\mu-\mu_k) = \mathcal{C}\mu - \mathcal{C}\mu_k = f-f=0.$$
Since $\Gamma_k:=\Gamma \cap R_k$ has area zero, we obtain from \cite[Corollary 1.3, Chapter 2]{GAR} that $\mu=\mu_k$ on $\Gamma_k$, i.e.
$$d\mu(\zeta)=(1/2\pi i)(f_{+}(\zeta)-f_{-}(\zeta))d\zeta$$
on $\Gamma_k$. This holds for all $k \in \mathbb{N}$, and therefore
$$d\mu(\zeta)=(1/2\pi i)(f_{+}(\zeta)-f_{-}(\zeta))d\zeta$$
on $\Gamma$.
\\

Now, recall from the introduction that since $E$ is connected, the Ahlfors function $f$ is the Riemann map of $\Omega$ onto $\mathbb{D}$ normalized by $f(\infty)=0$ and $f'(\infty)>0$. Note that the point $0 \in E$ defines two different accessible boundary points, and thus it follows from the correspondence of boundaries under Riemann maps (see e.g. \cite[Section 3, Chapter 2]{GOL}) that $f_{+}(\zeta)-f_{-}(\zeta) \to e^{i\theta_1}-e^{i\theta_2}$ as $\zeta \to 0$ on $\Gamma$, for some distinct points $e^{i\theta_1},e^{i\theta_2}$ on $\mathbb{T}$. Hence, if $k$ is sufficiently large, then
$$|f_{+}(\zeta)-f_{-}(\zeta)| \geq |e^{i\theta_1}-e^{i\theta_2}|/2 \qquad (\zeta \in \Gamma \setminus \Gamma_k)$$
and so
$$\|\mu\| \geq \frac{1}{2\pi} \int_{\Gamma \setminus \Gamma_k} |f_{-}(\zeta)-f_{+}(\zeta)| |d\zeta| \geq \frac{|e^{i\theta_1}-e^{i\theta_2}|}{4\pi} \int_{\Gamma \setminus \Gamma_k} |d\zeta| = \infty,$$
because $\Gamma$ is nonrectifiable whereas $\Gamma_k$ is. This contradiction shows that the Ahlfors function $f$ is not the Cauchy transform of any complex Borel measure supported on $E$.
\\

To complete the proof, it remains to show that $\gamma(E) = \gamma_c(E)$.
\\

For $k \in \mathbb{N}$, let $E_k$ be the union of $\Gamma \cap \overline{R_k}$ with the line segment $[-i,i]$. Then $(E_k)_{k \in \mathbb{N}}$ is an increasing sequence of compact subsets of $E$.

\begin{figure}[h!t!b]
\begin{center}
\includegraphics[width=6cm, height=6cm]{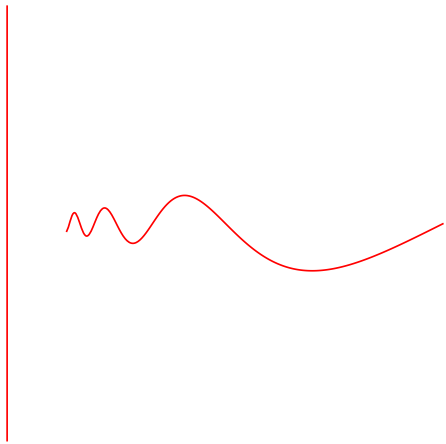}
\caption{The compact set $E_k$.}
\end{center}
\end{figure}
We claim that $\gamma(E_k) \to \gamma(E)$ as $k \to \infty$. By monotonicity of analytic capacity, it suffices to prove that there is a subsequence converging to $\gamma(E)$. Let $\Omega_k:= \mathbb{C}_\infty \setminus E_k$, so that $\Omega_k$ is a nondegenerate doubly connected domain. Then it is easy to see that $\Omega_k \to \Omega$. For $k \in \mathbb{N}$, let $g_k:\Omega_k \to \Omega_k'$ be the normalized Koebe map of $\Omega_k$. By Lemma \ref{lem normal}, the sequence $(g_k)_{k \in \mathbb{N}}$ has a subsequence converging locally uniformly to a univalent function $g$ on $\Omega$. By Theorem \ref{thm cara}, the corresponding subsequence of circle domains $(\Omega_k')$ converges to $g(\Omega)$ in the sense of Carath\'eodory. Note then that the sequences of centers and radii of the circles bounding the domains $\Omega_k'$ must be bounded, otherwise the kernel of the subsequence $(\Omega_k')$ would not contain a neighborhood of infinity. Therefore, passing to a subsequence if necessary, we can assume that the circles and radii of the circles bounding the domains $\Omega_k'$ converge to $c_1,c_2$ and $r_1,r_2$ respectively. Then it is easy to see that $g(\Omega)$ must be the domain bounded by the circles centered at $c_1,c_2$ of radius $r_1,r_2$ respectively. Now, since $g$ is univalent on $\Omega$, the image $g(\Omega)$ must be a nondegenerate simply connected domain and thus it follows that these circles must intersect at exactly one point. Hence by Lemma \ref{lem 2disks}, we get that
$$\gamma(\mathbb{C}_\infty \setminus \Omega_k') \to \gamma(\mathbb{C}_\infty \setminus g(\Omega))$$
along the subsequence. Again, a simple change of variable argument shows that
$$\gamma(\mathbb{C}_\infty \setminus \Omega_k')=\gamma(\mathbb{C}_\infty \setminus g(\Omega_k))=\gamma(\mathbb{C}_\infty \setminus \Omega_k) = \gamma(E_k)$$
and
$$\gamma(\mathbb{C}_\infty \setminus g(\Omega))=\gamma(\mathbb{C}_\infty \setminus \Omega)= \gamma(E).$$
This proves that $\gamma(E_k) \to \gamma(E)$.
\\

Finally, since every $E_k$ has finite Painlev\'e length, we obtain
$$\gamma_c(E) \leq \gamma(E)= \lim_{k \to \infty} \gamma(E_k) = \lim_{k \to \infty} \gamma_c(E_k) \leq \gamma_c(E),$$
where we used Corollary \ref{coro1} and the monotonicity of $\gamma_c$. This completes the proof of the theorem.

\end{proof}

As a consequence, we obtain a negative answer to Question \ref{prob2}.

\begin{corollary}
Let $E$ be as in Theorem \ref{thm0}. Then there is no complex Borel measure $\mu$ supported on $E$ such that
$$|\mathcal{C}_\mu(z)| \leq 1 \qquad (z \in \mathbb{C}_\infty \setminus E)$$ and
$\mu(E)=\gamma_c(E)$.
\end{corollary}

\begin{proof}
Suppose that such a measure $\mu$ exists. Then $\mathcal{C}(-\mu)$ is a function holomorphic on $\mathbb{C}_\infty \setminus E$ with $|\mathcal{C}(-\mu)|\leq 1$ and $$\mathcal{C}(-\mu)'(\infty)=\mu(E)=\gamma_c(E)=\gamma(E).$$
By uniqueness of the Ahlfors function $f$, it follows that $f= \mathcal{C}(-\mu)$ on $\mathbb{C}_\infty \setminus E$, which contradicts Theorem \ref{thm0}.

\end{proof}

\section{A convergence theorem for Ahlfors functions}
\label{sec2}
In this section, we state and prove a convergence theorem for Ahlfors functions on which relies Theorem \ref{thm2}. The arguments below are similar to the ones used in a paper of Fortier Bourque and the author \cite{FOR} for the proof of a theorem on rational Ahlfors functions.
\\

First, we need a convergence lemma for normalized Koebe maps.

\begin{lemma}
\label{conv Koebe}
Let $(\Omega_k)$ be a sequence of nondegenerate $m$-connected domains, each containing the point $\infty$, such that $\Omega_k \to \Omega$, where $\Omega$ is a nondegenerate $m$-connected domain. For each $k$, let $g_k : \Omega_k \to \Omega_k'$ be the normalized Koebe map of $\Omega_k$.
\\

Then $(g_k)_{k \in \mathbb{N}}$ converges locally uniformly on $\Omega$ to the normalized Koebe map $g:\Omega \to \Omega'$. In particular, the nondegenerate circle domains $\Omega_k'$ converge to $\Omega'$ in the sense of Carath\'eodory.
\end{lemma}

\begin{proof}
By Lemma \ref{lem normal}, every subsequence of $(g_k)_{k \in \mathbb{N}}$ has a subsequence that converges locally uniformly to a univalent function on $\Omega$.
\\

Let $h$ be a locally uniform limit of a subsequence $(g_k)_{k \in S}$. Then $h$ is normalized at infinity. Furthermore, by Theorem \ref{thm cara}, the corresponding subsequence of circle domains $(\Omega_k')_{k \in S}$ converge to $h(\Omega)$ in the sense of Carath\'eodory.
\\

We claim that $h(\Omega)$ is a nondegenerate circle domain, so that $h=g$ by uniqueness of the normalized Koebe map. Indeed, first note that $h(\Omega)$ is a nondegenerate $m$-connected domain since $h$ is univalent on $\Omega$. Moreover, the sequences of centers and radii of the circles bounding $\Omega_k'$ for $k \in S$ must be bounded, otherwise the kernel of $(\Omega_k')_{k \in S}$ would not contain a neighborhood of $\infty$. Passing to a subsequence if necessary, we can therefore assume that the centers of the circles bounding $(\Omega_k')_{k \in S}$ converge to $c_1,\dots,c_m \in \mathbb{C}$ and that the corresponding radii converge to $r_1,\dots,r_m \in \mathbb{R}$. Since $\Omega_k' \to h(\Omega)$ as $k \to \infty$ in $S$, clearly $h(\Omega)$ is the domain bounded by the circles centered at $c_1,\dots,c_m$ and of corresponding radii $r_1,\dots,r_m$. But $h(\Omega)$ is a nondegenerate $m$-connected domain, so these circles must be disjoint and nondegenerate. In other words, $h(\Omega)$ is a nondegenerate circle domain, from which it follows that $h=g$.
\\

This shows that every subsequence of $(g_k)_{k \in \mathbb{N}}$ has a subsequence that converges to $g$, which of course implies that $g_k \to g$.
\\

Finally, the fact that $\Omega_k' \to \Omega'$ is a consequence of Theorem \ref{thm cara}.
\end{proof}

\begin{remark}
A similar argument shows that if $D_k, D$ are nondegenerate $m$-connected circle domains, then $D_k \to D$ if and only if the sequences of centers and radii of the circles bounding the domains $D_k$ converge to the centers and radii of the circles bounding $D$.
\end{remark}

\begin{remark}
Lemma \ref{conv Koebe} is false without the assumption that the connectivity of $\Omega$ and of the $\Omega_k$'s is the same. Indeed, consider a sequence $(\Omega_k)$ of doubly connected domains bounded by two circles of radius one that get arbitrarily close to each other, so that the limit domain $\Omega$ is the complement of two closed disks of radius one intersecting at exactly one point. Then the normalized Koebe maps $g_k$'s are the identity maps, which clearly don't converge to the normalized Koebe map of $\Omega$.
\end{remark}

The following convergence theorem for Ahlfors functions is the main result of this section.

\begin{theorem}
\label{conv Ahlfors}
Let $(\Omega_k)$ be a sequence of nondegenerate $m$-connected domains, each containing the point $\infty$, such that $\Omega_k \to \Omega$, where $\Omega$ is a nondegenerate $m$-connected domain. Let $f_k,f$ be the Ahlfors functions on $\Omega_k$ and $\Omega$ respectively.
\\

Then $f_k \to f$ locally uniformly on $\Omega$.
\end{theorem}

\begin{proof}
Let $g_k:\Omega_k \to \Omega_k'$, $g:\Omega \to \Omega'$ be the normalized Koebe maps of $\Omega_k$ and $\Omega$ respectively. By Lemma \ref{conv Koebe}, $g_k \to g$ locally uniformly on $\Omega$ and $\Omega_k' \to \Omega'$. Let $\phi_k,\phi$ be the Ahlfors functions on $\Omega_k'$ and $\Omega'$ respectively. We claim that $\phi_k \to \phi$ locally uniformly on $\Omega'$. Let us prove this by showing that every subsequence of $(\phi_k)_{k \in \mathbb{N}}$ has a subsequence that converges to $\phi$.
\\

First note that by Montel's theorem, every subsequence of $(\phi_k)_{k \in \mathbb{N}}$ has a locally uniformly convergent subsequence. Let $\tilde{\phi}$ be the limit of a subsequence. Then $\tilde{\phi}$ is holomorphic on $\Omega'$ and satisfies $|\tilde{\phi}|\leq 1$, so we have $0 \leq \tilde{\phi}'(\infty) \leq \phi'(\infty)$.
\\

Now, by the Schwarz reflection principle, the function $\phi$ extends analytically to a neighborhood of $\overline{\Omega'}$ (recall that the Ahlfors function on a finitely connected domain is a proper map). Let $U$ be an open set containing $\overline{\Omega'}$ on which $\phi$ is bounded. By the remark following Lemma \ref{conv Koebe}, we have that $\Omega_k'$ is contained in $U$ for all $k$ sufficiently large, so that $\phi$ is defined and holomorphic on $\Omega_k'$ for these $k$'s. Let $M_k:=\sup\{|\phi(z)|:z \in \Omega_k'\}$. Then clearly $M_k \to 1$ as $k \to \infty$. Moreover, the function $M_k^{-1}\phi$ is holomorphic on $\Omega_k'$ and satisfies $|M_k^{-1}\phi|\leq 1$, so we have $M_k^{-1}\phi'(\infty) \leq \phi_k'(\infty)$. Therefore,
$$\phi'(\infty) \leq \liminf_{k \to \infty} \phi_k'(\infty) \leq \tilde{\phi}'(\infty)$$
and so $\tilde{\phi}'(\infty)=\phi'(\infty)$, which implies that $\tilde{\phi}=\phi$ by uniqueness of the Ahlfors function.
\\

This shows that $\phi_k \to \phi$ locally uniformly on $\Omega'$.
\\

Finally, a simple change of variable argument shows that $f=\phi \circ g$ and $f_k=\phi_k \circ g_k$ for all $k$. Since $g_k \to g$ locally uniformly on $\Omega$ and $\phi_k \to \phi$ locally uniformly on $\Omega'$, it follows that $f_k \to f$ locally uniformly on $\Omega$.

\end{proof}

\begin{remark}
Theorem \ref{conv Ahlfors} is false without any assumption on the connectivity of the domains $\Omega_k$, even if the limit domain $\Omega$ is assumed to be bounded by analytic curves. Indeed, consider a sequence $(z_k)$ dense in the unit disk $\mathbb{D}$ and for $k \in \mathbb{N}$, let $\Omega_k$ be the complement in the Riemann sphere of disjoint closed disks centered at $z_1,\dots,z_k$, contained in $\mathbb{D}$ and of radii sufficiently small so that the analytic capacity of $\mathbb{C}_\infty \setminus \Omega_k$ is less than $1/2$. Such a sequence exists by the outer regularity of analytic capacity and by the fact that the analytic capacity of a finite set is zero. Then it is easy to see that $\Omega_k \to \mathbb{C}_\infty \setminus \overline{\mathbb{D}}$ as $k \to \infty$, but the corresponding Ahlfors functions do not converge locally uniformly to the Ahlfors function on $\mathbb{C}_\infty \setminus \overline{\mathbb{D}}$, because otherwise we would have $\gamma(\mathbb{C}_\infty \setminus \Omega_k) \to \gamma(\overline{\mathbb{D}})=1.$ This example was mentioned to the author by Maxime Fortier Bourque.
\end{remark}

\section{Proof of Theorem \ref{thm2}}
\label{sec3}
We can now proceed with the proof of Theorem \ref{thm2}.

\begin{proof}

Let $E \subseteq \mathbb{C}$ be compact and suppose that there exist a sequence $(E_k)_{k \in \mathbb{N}}$ of compact subsets of $E$ with the following properties :

\begin{enumerate}[\rm(i)]
\item [(i)] every $E_k$ has finite Painlev\'e length;
\item [(ii)] there exist an integer $m$ such that $\Omega$ and every $\Omega_k$ are nondegenerate $m$-connected domains, where $\Omega_k$ and $\Omega$ are the unbounded components of $\mathbb{C}_\infty \setminus E_k$ and $\mathbb{C}_\infty \setminus E$ respectively;
\item [(iii)] the sequence of domains $(\Omega_k)_{k \in \mathbb{N}}$ converges to $\Omega$ in the sense of Carath\'eodory.
\end{enumerate}
We have to show that $\gamma(E)=\gamma_c(E)$.
\\

Let $f_k,f$ be the Ahlfors functions on $\Omega_k$ and $\Omega$ respectively. By Theorem \ref{conv Ahlfors}, $f_k \to f$ locally uniformly on $\Omega$, so in particular $f_k'(\infty) \to f'(\infty)$. Thus, we obtain
$$\gamma_c(E) \leq \gamma(E) = f'(\infty)=\lim_{k \to \infty} f_k'(\infty)= \lim_{k \to \infty} \gamma(E_k) = \lim_{k \to \infty} \gamma_c(E_k) \leq \gamma_c(E),$$
where we used Corollary \ref{coro1} and the monotonicity of $\gamma_c$. This completes the proof of the theorem.

\end{proof}

\section{Proof of Theorem \ref{thm1}}
\label{sec4}
In this section, we describe the construction of the compact sets of Theorem \ref{thm1}.
\\

More precisely, for every $m \in \mathbb{N}$, we shall construct a compact set $E^m$ with $m$ nondegenerate components such that $\gamma(E^m)=\gamma_c(E^m)$ but the Ahlfors function is not the Cauchy transform of any complex Borel measure supported on $E^m$. The construction is similar to the one in Sect. \ref{sec1}.

\begin{proof}
Consider first the case $m=1$. Let $E^1$ be the union of the curve
$$\Gamma:=\{ x+ix\operatorname{sin}(1/x): x \in (0,1/\pi] \}$$
with the line segments $[-i,0]$, $[-i,1/\pi -i]$ and $[1/\pi -i, 1/\pi]$. Then $E^1$ is compact and connected. Let $f$ be the Ahlfors function for $E^1$. Recall that $f$ is assumed to be identically zero in the bounded component of $\mathbb{C}_\infty \setminus E^1$.

\begin{figure}[h!t!b]
\begin{center}
\includegraphics[width=6cm, height=6cm]{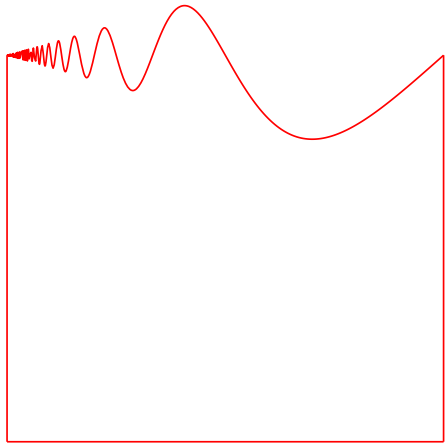}
\caption{The compact set $E^1$.}
\end{center}
\end{figure}

First, we claim that $f$ is not the Cauchy transform of any complex Borel measure supported on $E^1$. Indeed, suppose that $f=\mathcal{C}_\mu$ for such a $\mu$. Let $\Omega_{+}$ and $\Omega_{-}$ denote the upper and lower parts respectively of the complement of $\Gamma$ in the strip $\{z: 0<\operatorname{Re}z<1/\pi\}$. For $k \in \mathbb{N}$, let $R_k$ be the open rectangle
$$R_k:= \left\{z=x+iy : \frac{1}{\pi+k} <x<1/\pi, -1/2<y<1/2\right\}.$$
\newpage

\begin{figure}[h!t!b]
\begin{center}
\includegraphics[width=6cm, height=6cm]{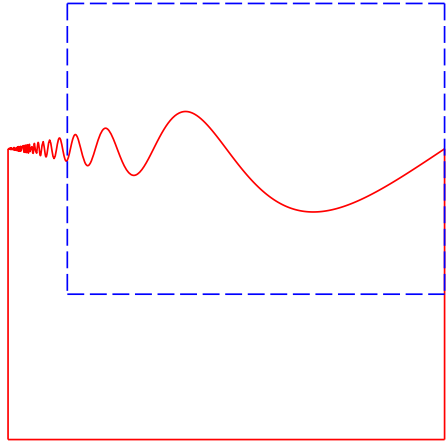}
\caption{The rectangle $R_k$.}
\end{center}
\end{figure}

Clearly, the restriction of $f$ to $R_k \cap \Omega_{+}$ is the Cauchy transform of a measure $\mu_k$ defined by $1/(2\pi i) f_{+}(\zeta) d\zeta$ on $\Gamma \cap R_k$ and $1/(2\pi i) f(\zeta) d\zeta$ on the linear parts of the boundary.

It follows that on $R_k \cap \Omega_{+}$, we have
$$\mathcal{C}_{\mu - \mu_k} = \mathcal{C}_{\mu} - \mathcal{C}_{\mu_k} = f-f=0.$$
On the other hand, on $R_k \cap \Omega_{-}$, $\mathcal{C}_{\mu_k}=0$ by Cauchy's theorem and $\mathcal{C}_\mu = f =0$, so that
$$\mathcal{C}_{\mu-\mu_k} = \mathcal{C}_{\mu} - \mathcal{C}_{\mu_k} = 0.$$
Since $\Gamma \cap R_k$ has area zero, it follows from \cite[Corollary 1.3, Chapter 2]{GAR} that $\mu=\mu_k$ on $\Gamma \cap R_k$, i.e.
$$d\mu(\zeta)=(1/2\pi i)f_{+}(\zeta) d\zeta$$
on $\Gamma \cap R_k$. This holds for all $k \in \mathbb{N}$, and so $d\mu(\zeta)=(1/2\pi i)f_{+}(\zeta) d\zeta$ on $\Gamma$. This gives a contradiction, since
$$\|\mu\| \geq \int_{\Gamma} |d\mu(\zeta)| = \frac{1}{2\pi} \int_{\Gamma} |f_{+}(\zeta)||d\zeta| = \frac{1}{2\pi} \int_{\Gamma}|d\zeta| = \infty,$$
because $\Gamma$ is nonrectifiable and $|f| \equiv 1$ on $E^1$, by properness of the Ahlfors function.
\\

Let us prove now that $\gamma(E^1)=\gamma_{c}(E^1)$. For $k \in \mathbb{N}$, let $E^1_k$ be the union of the portion of $\Gamma$ inside $\overline{R_k}$ with the line segments $[-i,0]$, $[-i,1/\pi -i]$ and $[1/\pi -i, 1/\pi]$. Then $E^1_k$ is a connected compact subset of $E^1$.

\newpage

\begin{figure}[h!t!b]
\begin{center}
\includegraphics[width=6cm, height=6cm]{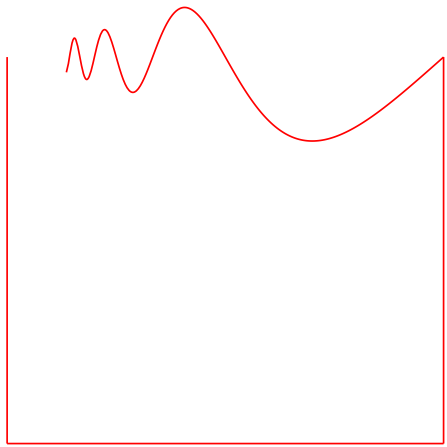}
\caption{The compact set $E^1_k$.}
\end{center}
\end{figure}
Furthermore, since $E^1_k$ has finite length, it has finite Painlev\'e length. Also, it is easy to see that $\Omega_k \to \Omega$, where $\Omega_k$ and $\Omega$ are the unbounded components of $\mathbb{C}_\infty \setminus E^1_k$ and $\mathbb{C}_\infty \setminus E^1$ respectively. By Theorem \ref{thm2}, $\gamma_c(E^1)=\gamma(E^1)$.
\\

Finally, if $m>1$, let $E^m$ be the union of $E^1$ with disjoint nondegenerate connected full compact sets $F_1, \dots, F_{m-1}$ of finite Painlev\'e length contained in the unbounded component of $\mathbb{C}_\infty \setminus E^1$. A simple modification of the above argument (with the sets $E^1_k$ replaced by $E^1_k \cup F_1 \cup \dots \cup F_{m-1}$) shows that $\gamma(E^m)=\gamma_c(E^m)$ and that the Ahlfors function for $E^m$ is not the Cauchy transform of any measure supported on $E^m$.
\end{proof}

\begin{remark}
We were not able to find examples of such compact sets that are disconnected and have connected complement. A natural idea is to replace the set $E^1$ in the above proof by the compact set $E$ of Theorem \ref{thm0}, but then the difficulty is in showing that the analytic capacity and the cauchy capacity of the resulting compact set are equal. In order to prove this, one would need a generalization of Lemma \ref{lem 2disks} for the convergence of the analytic capacities of $m$ disjoint closed disks where one pair of disks intersect at one point in the limit. This seems to be true but we are not able to prove it.
\end{remark}

\begin{remark}
For the compact set $E^m$ as above, there are other functions in $H^{\infty}(\mathbb{C}_\infty \setminus E^m)$ bounded by one in modulus whose derivatives at infinity are equal to $\gamma(E^m)$; it suffices to consider any function equal to the Ahlfors function $f$ in the unbounded component of $\mathbb{C}_\infty \setminus E^m$ and equal to an arbitrary holomorphic function $g$ with $|g| \leq 1$ in the bounded component of $\mathbb{C}_\infty \setminus E^m$. The above proof shows that if $g$ is identically zero, then the resulting function is not the Cauchy transform of any measure supported on $E^m$. Xavier Tolsa raised the question of whether this is true for any choice of $g$. A simple modification of the above proof shows that the answer is positive provided that the integral
$$\int_{\Gamma}|f_+(\zeta)-g_-(\zeta)||d\zeta|$$
diverges. This holds for instance if $g$ stays at a positive distance from $f_{+}(0)$ near the point $0$.
\end{remark}

\begin{remark}
The arguments used in this paper could be considerably simplified if one could prove that analytic capacity is \textit{inner regular}, in the sense that if $E_1 \subseteq E_2 \subseteq \dots$ is an increasing sequence of compact sets and if $E:=\cup_k E_k$ is compact, then
$$\gamma(E_k) \to \gamma(E).$$
This seems to be a open problem. It is closely related to the so-called \textit{capacitability} of analytic capacity.
\end{remark}

\acknowledgments{The author thanks Xavier Tolsa for helpful discussions.}

\bibliographystyle{amsplain}

\begin{thebibliography}{99}

\bibitem{AHL}
L. Ahlfors,
Bounded analytic functions,
\textsl{Duke Math. J.}
\textbf{14} (1947),
1--11.

\bibitem{CON}
J.B. Conway,
\textsl{Functions of one complex variable. {II}},
Springer-Verlag, New-York,
1995.

\bibitem{DAV}
G. David,
Unrectifiable $1$-sets have vanishing analytic capacity,
\textsl{Rev. Mat. Iberoamericana},
\textbf{14} (1998),
369--479.

\bibitem{FOR}
M. Fortier Bourque and M. Younsi,
Rational Ahlfors functions,
\textsl{To appear in Constr. Approx.},
DOI 10.1007/s00365-014-9260-4 (2014).

\bibitem{GAR}
J. Garnett,
\textsl{Analytic capacity and measure},
Springer-Verlag, Berlin,
1972.

\bibitem{GOL}
G.M. Goluzin,
\textsl{Geometric theory of functions of a complex variable},
American Mathematical Society, Providence,
1969.

\bibitem{HAV}
S.Ya. Havinson,
Representation and approximation of functions on thin sets (Russian),
\textsl{1966 Contemporary Problems in Theory Anal. Functions}
314--318.

\bibitem{HAV2}
S.Ya. Havinson,
Golubev sums: a theory of extremal problems that are of the analytic capacity problem type and of accompanying approximation processes (Russian),
\textsl{Uspekhi Mat. Nauk}
\textbf{54} (1999),
75--142.
Translation in \textsl{Russian Math. Surveys}
\textbf{54} (1999), 753--818.

\bibitem{HAV3}
S.Ya. Havinson,
Duality relations in the theory of analytic capacity (Russian),
\textsl{Algebra i Analiz}
\textbf{15} (2003),
3--62.
Translation in \textsl{St. Petersburg Math. J.}
\textbf{15} (2004), 1--40.

\bibitem{HAV4}
S. Ja. Havinson,
Analytic capacity of sets, joint nontriviality of various classes of analytic functions and the Schwarz lemma in arbitrary domains (Russian)
\textsl{Mat. Sb. (N.S.)}
\textbf{54 (96)} (1961), 3--50.
Translation in \textsl{Amer. Math. Soc. Transl. (2)}
\textbf{43} (1964), 215--266.

\bibitem{MEL}
M. Melnikov,
Analytic capacity: a discrete approach and the curvature of measure (Russian),
\textsl{Mat. Sb.}
\textbf{186} (1995),
57--76.
Translation in \textsl{Sb. Math.}
\textbf{186} (1995), 827--846.

\bibitem{MUR}
T. Murai,
Analyic capacity for arcs,
\textsl{Proceedings of the International Congress of Mathematicians}
\textbf{1} (1991),
901-911.

\bibitem{SAM}
M.V. Samokhin,
On the Cauchy integral formula in domains of arbitrary connectivity (Russian),
\textsl{Mat. Sb.}
\textbf{191} (2000),
113--130.
Translation in \textsl{Sb. Math.}
\textbf{191} (2000),
1215--1231.

\bibitem{TOL}
X. Tolsa,
\textsl{Analytic capacity, the Cauchy transform, and non-homogeneous Calder\'on-Zygmund theory}
Birkh\"hauser/Springer, Cham,
2014.

\bibitem{TOL2}
X. Tolsa,
Analytic capacity, rectifiability and the Cauchy integral,
\textsl{International Congress of Mathematicians}
\textbf{2} (2006),
1505--1527.

\bibitem{TOL3}
X. Tolsa,
Painlev\'e's problem and the semiadditivity of analytic capacity,
\textsl{Acta Math.}
\textbf{190} (2003),
105--149.

\bibitem{VIT}
A. Vitu{\v{s}}kin,
Analytic capacity of sets in problems of approximation theory (Russian),
\textsl{Uspehi Mat. Nauk}
\textbf{22} (1967),
141--199.

\bibitem{YOU}
M. Younsi and T. Ransford,
Computation of analytic capacity and applications to the subadditivity problem,
\textsl{Comput. Methods Funct. Theory}
\textbf{13} (2013),
337--382.


\end{thebibliography}

\end{document}